\documentclass[12pt]{article}
\usepackage{amssymb}
\usepackage{amsmath,amsthm}
\usepackage[latin1]{inputenc}
\usepackage{hyperref}
\usepackage{color}
\usepackage{enumerate}
\usepackage{graphicx}
 \newtheorem{remark}{Remark}

 \newtheorem{lemma}[remark]{Lemma}
 \newtheorem{theorem}[remark]{Theorem}
 \newtheorem{proposition}[remark]{Proposition}
 \newtheorem{corollary}[remark]{Corollary}

 \newtheorem{conjecture}[remark]{Conjecture}

\title{Computing global offensive alliances in Cartesian product graphs}

\author{ Ismael G. Yero$^{1}$ and Juan A.
Rodr\'{\i}guez-Vel\'{a}zquez$^{2}$\\
    \\
$^1${\small Departamento de Matem\'aticas, Escuela Polit\'ecnica Superior de Algeciras}\\
{\small Universidad de C\'adiz,} {\small
Av. Ram\'on Puyol s/n, 11202 Algeciras, Spain.} \\ {\small
ismael.gonzalez\@@uca.es}\\
$^2${\small Departament d'Enginyeria Inform\`atica i Matem\`atiques}\\
{\small Universitat Rovira i Virgili,}  {\small Av. Pa\"{\i}sos
Catalans 26, 43007 Tarragona, Spain.} \\{\small
juanalberto.rodriguez\@@urv.cat}
\\
}

\begin{document}

\maketitle

\begin{abstract}
A global offensive alliance in a graph $G$ is a set $S$ of vertices with
the property that every vertex not belonging to $S$ has at least one more
neighbor in $S$ than it has outside of $S$. The global offensive alliance number  of $G$, $\gamma_o(G)$, is the minimum cardinality of a global offensive alliance in $G$. A set $S$ of vertices of a graph $G$ is a dominating set for $G$ if
every vertex not belonging to $S$ has at least one neighbor in $S$. The domination number  of $G$, $\gamma(G)$, is the minimum cardinality of a dominating set of $G$.    In this work  we obtain closed formulas for the global offensive alliance number of several families of Cartesian product graphs, we also prove that $\gamma_o(G\square H)\ge \frac{\gamma(G)\gamma_o(H)}{2}$ for any graphs $G$ and $H$ and we show that if $G$ has an efficient dominating set, then $\gamma_o(G\square H)\ge \gamma(G)\gamma_o(H).$ Moreover, we present a Vizing-like conjecture for the global offensive alliance number and we prove it for several families of graphs.
\end{abstract}



\section{Introduction}



Alliances in graphs were described first  by Kristiansen {\em et al } in \cite{alliancesOne}, where  alliances were classified into defensive, offensive or powerful. After this seminal paper,  the issue has been studied intensively. Remarkable examples are the articles \cite{kdaf,kdaf1}, where alliances were generalized to $k$-alliances, and \cite{fava}, where the authors presented the first results on offensive alliances. One of the main motivations of this study  is based on the NP-completeness of computing minimum cardinality of (defensive, offensive, powerful) alliances in graphs.

On the other hand,  several graphs may be constructible from smaller and simpler components by basic operations like unions, joins, compositions, or multiplications with respect to various products, where properties
of the constituents determine the properties of the composite graph. It is therefore desirable to reduce the problem of computing the
graph parameters (alliance numbers, for instance) of product graphs, to the problem of
computing some parameters of the factor graphs.


Nowadays the study
of the behavior of several graph parameters in product graphs has become an interesting topic
of research \cite{Imrich,product2}. For instance, we emphasize the Shannon capacity of a graph \cite{shanon}, which is
a certain limiting value involving the vertex independence number of strong product powers of a
graph, and Hedetniemi's coloring conjecture for the categorical product \cite{hedet,product2}, which states
that the chromatic number of any categorial product graph is equal to the minimum value between
the chromatic numbers of its factors. Also, one of the oldest open problems on domination
in graphs is related to Cartesian product graphs. The problem was presented first by Vizing
in 1963 \cite{vizing1,vizing2}. Vizing's conjecture states that the domination number of any Cartesian product
graph is greater than or equal to the product of the domination numbers of its factors.

Cartesian product graphs have been much studied
in graph theory. Interest in Cartesian product graphs has been increased by the  advent of massively parallel computers whose structure is that of the Cartesian product graphs \cite{Ha,Ka}. This not only provides potential
applications for the existing theory, but also suggests some new aspects of these graphs that deserve study.

 Studies on defensive alliances in graphs product were initiated in
\cite{cota}, for the case of the torus graph $C_s\square C_t$, and studied further in \cite{bermudo,yero1,yero2}. Nevertheless the great part of the results in these works are upper bounds on the alliance numbers of Cartesian product graphs. In the present work we obtain closed formulas for the global offensive alliance number of several families of Cartesian product graphs, we obtain new formulas relating the global offensive alliance number of Cartesian product graphs with the domination number and the global offensive alliance number of its factors, we present a Vizing-like conjecture for the global offensive alliance number and we prove it for several families of graphs.

We begin by stating the terminology
used. Throughout this article, $G = (V,E)$ denotes a simple
graph of order $|V | = n$. We denote two adjacent vertices $u$ and
$v$ by $u \sim v$. Given a vertex $v\in V,$ the set $N(v)=\{u\in V: \; u\sim v\}$ is the open
neighborhood of $v$, and  the set  $N[v]=N(v)\cup \{v\}$ is the closed
neighborhood of $v$. So, the  degree of
a vertex $v \in V$ is  $d(v)=|N(v)|$.

For a nonempty set $S \subseteq V$, and a vertex
$v \in V$, $N_S(v)$ denotes the set of neighbors $v$ has in $S$, i.e.,
$N_S(v) = S\cap N(v)$. The degree of $v$ in $S$
will be denoted by $\delta_S(v) = |N_S(v)|$.  The complement of a
set $S$ in $V$ is denoted by $\overline{S}$.

A set $S \subseteq V$ is a \textit{dominating set} in $G$ if
for every vertex $v\in \overline{S}$,  $\delta_S(v)> 0$ (every
vertex in $\overline{S}$ is adjacent to at least one vertex in $S$).
The \textit{domination number} of $G$, denoted by
$\gamma(G)$, is the minimum cardinality of a dominating set in
$G$ \cite{bookdom1}. An \emph{efficient dominating set}  is a dominating set  $S=\{u_1,u_2,...,u_{{\gamma(G)}}\}$
such that $N[u_i]\cap N[u_j]=\emptyset$, for every $i,j\in
\{1,...,\gamma(G)\}$, $i\ne j$.   Examples of graphs having an efficient dominating set are the path graphs $P_n$, the cycle  graphs $C_{3k}$ and the cube graph $Q_3$.

A nonempty set $S
\subseteq V$ is a \textit{global offensive alliance} in $G$ if
\begin{equation}\label{alliaceCondition}\delta_S(v)\geq
\delta_{\overline{S}}(v)+1, \quad \forall v \in \overline{S}\end{equation} or, equivalently, \begin{equation}\label{alliaceCondition2}
d(v)\geq
2\delta_{\overline{S}}(v)+1,\quad \forall v \in \overline{S}.
\end{equation} Note that every global offensive alliance is a dominating set. The \textit{global offensive
alliance number} of $G$, denoted by $\gamma_o(G)$, is
defined as the minimum cardinality of a global offensive alliance in
$G$.  A global offensive alliance of cardinality $\gamma_o(G)$ is called a $\gamma_o(G)$-set.

We recall that given two graphs $G$ and $H$ with set of vertices
$V_1=\{v_1,v_2,...,v_{n_1}\}$ and $V_2=\{u_1,u_2,...,u_{n_2}\}$,
respectively, the Cartesian product of $G$ and $H$ is
the graph $G\square H=(V,E)$, where $V=V_1\times V_2$ and  two vertices
$(v_i,u_j)$ and $(v_k,u_l)$ are adjacent in $G\square H$ if and only if
\begin{itemize}
\item $v_i=v_k$ and $u_j\sim u_l$, or

\item $v_i\sim v_k$ and $u_j=u_l$.
\end{itemize}
Given two graphs $G=(V_1,E_1)$, $H=(V_2,E_2)$ and a set $X\subset V_1\times V_2$ of vertices of $G\Box H$, the projections of $X$ over $V_1$ and $V_2$  are denoted by
$P_G(X)$ and $P_H(X)$, respectively.
Moreover, given a set $C\subset V_1$ of vertices of $G$ and a vertex $v\in V_2$, a $G(C,v)$-{\em cell} in $G\Box H$ is the set $C^v=\{(u,v)\in V\,:\, u\in C\}$. A $v$-fiber $G_v$ is the copy of $G$ corresponding to the vertex $v$ of $H$. For every $v\in V_2$ and $D\subset V_1\times V_2$,  let $D_v$ be the set of vertices of $D$ belonging to the same $v$-fiber.

Now we establish a  Vizing-like conjecture for the global offensive alliance number.
\begin{conjecture}[{\rm  Vizing-like conjecture for the global offensive alliances}]
For any graphs $G$ and $H$,
$$\gamma_o(G\Box H)\ge \gamma_o(G)\gamma_o(H).$$
\end{conjecture}
Below we will prove the conjecture for several families of graphs.

\section{Results}
\begin{theorem}
For any graphs $G$ and $H$,
$$\gamma_o(G\Box H)\ge \frac{1}{2}\max\{\gamma(G)\gamma_o(H),\gamma_o(G)\gamma(H)\}.$$
Moreover, if $G$ has an efficient dominating set, then $$\gamma_o(G\Box H)\ge \gamma(G)\gamma_o(H).$$
\end{theorem}

\begin{proof}
Let $V_1$ and $V_2$ be the vertex sets of the graphs $G$ and $H$, respectively.
Let $S=\{u_1,...,u_{\gamma(G)}\}$ be a dominating set for $G$. Let $\Pi=\{A_1,A_2,...,A_{\gamma(G)}\}$ be a vertex
partition of $G$ such that $u_i\in A_i$ and $ A_i\subseteq N[u_i]$.
Let $\{\Pi_1,\Pi_2,...,\Pi_{\gamma(G)}\}$ be a vertex partition of
$G\Box H$, such that $\Pi_i=A_i\times V_2$ for every
$i\in\{1,...,\gamma(G)\}$.

Let $D$ be a  $\gamma_o(G\Box H)$-set. Now,
for every $i\in \{1,...,\gamma(G)\}$, let $W_i=P_H(D\cap \Pi_i)$.
We have
\begin{equation}\label{paraCaso2}
\gamma_o(G\Box H)=|D|\ge \sum_{i=1}^{\gamma(G)}|W_i|.
\end{equation}
If $W_i$ is not a global offensive alliance in $H$, then there exists at least a vertex $v\in \overline{W_i}$ such that
\begin{equation}\label{alliance-cond-1}
\delta_{W_i}(v)<\delta_{\overline{W_i}}(v)+1.
\end{equation}
So, every vertex belonging to  $A_i^v$ has at least one neighbor in $D_v\cap \Pi_j$, for some $j\ne i$. For every  $v\in V_2$, let $A_{j_1}^v,A_{j_2}^v,...,A_{j_{q_v}}^v$ be the $G(A_{j_i},v)$-cells for which $v$ satisfies (\ref{alliance-cond-1}) and let $Y_v=S-\{u_{j_1},u_{j_2},...,u_{j_{q_v}}\}$. Since $Y_v$ dominates $V_1\times \{v\}-\bigcup_{i=1}^{q_v} A_{j_i}^v$ and $D_v$ dominates  $\bigcup_{i=1}^{q_v} A_{j_i}^v$, we have that $S_v=D_v\cup Y_v$ is a dominating set in the $v$-fiber $G_v$.

Now, for every $i\in \{1,...,\gamma(G)\}$, let $Q_i\subseteq \overline{W_i}$ be the set of vertices of $H$ satisfying the inequality (\ref{alliance-cond-1}). Since  $W_i\cup Q_i$ is a global offensive alliance in $H$,
\begin{equation}\label{paraCaso2otro}
\gamma_o(H)\le |W_i|+|Q_i|.
\end{equation}

Hence, we have that
\begin{align*}
\gamma_o(G\Box H)&\ge \sum_{i=1}^{\gamma(G)}|W_i|\\
&\ge\sum_{i=1}^{\gamma(G)}(\gamma_o(H)-|Q_i|)\\
&=\gamma(G)\gamma_o(H)-\sum_{i=1}^{\gamma(G)}|Q_i|,
\end{align*}

and, as a consequence, we have

\begin{equation}\label{cond-1}
\gamma_o(G\Box H)\ge \gamma(G)\gamma_o(H)-\sum_{i=1}^{\gamma(G)}|Q_i|.
\end{equation}

On the other hand, notice
that for each $v\in V_2$, $q_v$ is the number of $G(A_i,v)$-cells for which $v$ satisfies (\ref{alliance-cond-1}), as well as for each $i\in \{1,...,\gamma(G)\}$, $|Q_i|$ is the number of vertices of $H$ satisfying inequality (\ref{alliance-cond-1}). Thus,
\begin{equation}\label{sumasiguales}
\displaystyle\sum_{v\in V_2}q_v=\displaystyle\sum_{i=1}^{\gamma(G)}|Q_i|.
\end{equation}

 Now, if $q_v > |D_v|$, then we have
\begin{align*}
|S_v|&=|D_v|+|Y_v|\\
&=|S|-q_v+|D_v|\\
&=\gamma(G)-q_v+|D_v|\\
&<\gamma(G)-q_v+q_v\\
&=\gamma(G),
\end{align*}
which is a contradiction. So, we have $q_v \le |D_v|$
and we obtain
\begin{equation}\label{cond-2}
\sum_{v\in V_2}q_v\le \sum_{v\in V_2}|D_v|=\gamma_o(G\Box H),
\end{equation}
Thus, by  (\ref{cond-1}), (\ref{sumasiguales}) and (\ref{cond-2}) we deduce
$$\gamma_o(G\Box H)\ge \gamma(G)\gamma_o(H)-\gamma_o(G\Box H).$$
Analogously, we obtain that $\gamma_o(G\Box H)\ge \gamma_o(G)\gamma(H)-\gamma_o(G\Box H)$. Therefore, the first result follows.

Now, if $S=\{u_1,...,u_{\gamma(G)}\}$ is  an efficient  dominating set for $G$, then for every $i\in \{1,...,\gamma(G)\}$, $W_i$ is a global offensive alliance in $H$. That is, if we suppose that $W_i$ is not a global offensive alliance in $H$, then there exists at least one vertex $v\in \overline{W_i}$ which satisfies (\ref{alliance-cond-1}). Thus, every vertex belonging to $A_i^v$ has at least one neighbor in $D_v\cap \Pi_j$, for some $j\ne i$, which is a contradiction because  $(u_i,v)$ has no neighbors outside of $\Pi_i$. As  a consequence,  $|Q_i|=0$. So, (\ref{paraCaso2}) and (\ref{paraCaso2otro}) directly lead to $\gamma_o(G\Box H)\ge \gamma(G)\gamma_o(H)$.
\end{proof}

Notice that for the case of star graphs, $S_{1,n}$, the central vertex forms an efficient dominating set of minimum cardinality, and it is also a global offensive alliance, then the above theorem leads to the following Vizing-like result for the global offensive alliance number.

\begin{corollary}
Let $S_{1,n}$ be a star graph. For any graph $H$,
$$\gamma_o(S_{1,n}\Box H)\ge \gamma_o(S_{1,n})\gamma_o(H).$$
\end{corollary}

As the following remark shows there is no other family of connected graphs containing an efficient dominating set of minimum cardinality which is also a global offensive alliance.

\begin{remark}
A connected graph  $G$ contains an efficient dominating set of minimum cardinality which is also a global offensive alliance if and only if $G$ is a star graph.
\end{remark}

\begin{proof}
If $G$ is a star graph, then it is clear that the central vertex is an efficient dominating set of minimum cardinality, and also a global offensive alliance.

On the contrary, suppose $G$ is not a star graph and let $S=\{u_1,u_2,...,u_{\gamma(G)}\}$ be an efficient dominating set of minimum cardinality which is also a global offensive alliance in $G$. So, for every $u_i,u_j\in S$, $i\ne j$, we have that $N[u_i]\cap N[u_j]=\emptyset$. As a consequence, for every $v\in \overline{S}$ we have $\delta_S(v)=1$ and by inequality (\ref{alliaceCondition}) we have,
$\delta_S(v)\ge \delta_{\overline{S}}(v)+1.$
Hence, $\delta_{\overline{S}}(v)=0$ and so the degree of $v$ in $G$ is one, {\em i.e.}, every vertex outside of $S$ is an end-vertex.
Now, if $\gamma(G)\ge 2$, then since for every $u_i,u_j\in S$, $i\ne j$, we have that $N[u_i]\cap N[u_j]=\emptyset$ we obtain that $u_i\not\sim u_j$. So, $G$ is not connected, which is a contradiction. Therefore  $\gamma(G)= 1$ and, as a consequence, $G$ is a star graph.
\end{proof}

\begin{theorem}\label{ProdGxPn}
Let $P_{n}$ be a path graph of order $n$. For every graph $G$ of minimum degree $\delta\ge 1$,
$$\gamma_o(G\Box P_n)\ge  \left\lceil\frac{(n-1)\gamma_o(G)}{2}\right\rceil+\left\lceil\frac{\delta}{2}\right\rceil. $$
\end{theorem}

\begin{proof}
 Let $S$ be a $\gamma_o(G\Box P_{n})$-set. Let $\{v_1,v_2,...,v_{n}\}$ be the set of vertices of $P_{n}$. Let $V_i$ be the vertex set of the $v_i$-fiber $G_{v_i}$ and let $S_i=P_G(S\cap V_i)$.

 For every $(x,v_1)\not\in S$ we have $ \delta_S(x,v_1)\ge \delta_{\overline{S}}(x,v_1)+1$. Adding $\delta_S(x,v_1)$ to both sides of this inequality we have  $2\delta_S(x,v_1)\ge d(x,v_1)+1$. Hence,
 $$2(|S_1|+1) \ge 2(\delta_{S_1}(x)+1)\ge 2\delta_S(x,v_1)\ge d(x,v_1)+1=d(x)+2\ge \delta+2.$$ As a consequence,  $|S_1|\ge \left\lceil\frac{\delta}{2}\right\rceil$.
 Analogously we show that $|S_n|\ge \left\lceil\frac{\delta}{2}\right\rceil$.

We suppose there exists $i\in \{1,2,...,n-1\}$ such that $S_i$ is not a global offensive alliance in $G$. Let $S'_i\subset V_i-S_i$  such that $\delta_{S_i}(x)< \delta_{\overline{S_i}}(x)+1$, for every $x\in S'_i$. Now let $x\in S'_i$ and suppose $(x,v_{i+1})\not\in S$. If $i=1$, then  $\delta_S(x,v_1)=\delta_{S_1}(x)<\delta_{\overline{S_1}}(x)+1=\delta_{\overline{S}}(x,v_1),$ a contradiction.  If $1<i<n$, then $\delta_S(x,v_i)\le \delta_{S_i}(x)+1<\delta_{\overline{S_i}}(x)+2\le \delta_{\overline{S}}(x,v_i)+1,$ also a contradiction. Hence, if  $S_i$ is not a global offensive alliance in $G$, then for every $x\in S'_i$ we have  $(x,v_{i+1})\in S$. As a consequence, $S_i\cup S'_i$ is a global offensive alliance in $G$ and for every $i\in \{1,2,...,n-1\}$, $|S_i\cup S_{i+1}|\ge |S_i\cup S'_i|\ge \gamma_o(G)$.
So,
$ \sum_{i=1}^{{n}-1}| S_i\cup S_{i+1}|\ge ({n}-1)\gamma_o(G)$ and we have
 $$ 2|S|=2\sum_{i=1}^{n}|S_i|=\sum_{i=1}^{n-1}|S_i\cup S_{i+1}| + |S_1|+|S_n|\ge (n-1)\gamma_o(G)+|S_1|+|S_{n}|\ge (n-1)\gamma_o(G)+2 \left\lceil\frac{\delta}{2}\right\rceil.$$
  Therefore,
$|S|\ge \frac{(n-1)\gamma_o(G)}{2}+\left\lceil\frac{\delta}{2}\right\rceil.$
\end{proof}

We note that since $\gamma_o(P_{n})=\left\lfloor\frac{n}{2}\right\rfloor$, the above theorem leads to the following corollary.
\begin{corollary}
Let $P_{2k+1}$ be a path graph of odd order. For any graph $G$, $$\gamma_o(G\Box P_{2k+1})> \gamma_o(G)\gamma_o(P_{2k+1}).$$
\end{corollary}

\begin{theorem}\label{GporCiclos}
Let $C_{n}$ be a cycle graph of order $n$. For every graph $G$,
$$\gamma_o(G\Box C_n)\ge \left\lceil \frac{n\gamma_o(G)}{2}\right\rceil.$$
\end{theorem}

\begin{proof}
Let $S$ be a $\gamma_o(G\Box C_{n})$-set. Let $\{v_0,v_2,...,v_{n-1}\}$ be the set of vertices of $C_{n}$. Let $V_i$ be the vertex set of the $v_i$-fiber $G_{v_i}$ and let $S_i=P_G(S\cap V_i)$.
Proceeding like in the proof of Theorem \ref{ProdGxPn} we show that for every $i\in \{0,1,...,n-1\}$, $|S_i\cup S_{i+1}|\ge \gamma_o(G)$, where the subscripts are taken modulo $n$. Hence,  $$2|S|=\sum_{i=0}^{{n-1}}| S_i\cup S_{i+1})|\ge n\gamma_o(G).$$ Therefore, the  result follows.
\end{proof}


In order to deduce an upper bound on $\gamma_o(G\Box H)$ we are going to introduce two known results. It was shown in \cite{cota-global} that for every connected graph $G$ of order $n\ge 2$ and independence number $\alpha(G)$, it follows that
\begin{equation}\label{CotaOffInd}
\gamma_o(G)+\alpha(G)\le n.
\end{equation}
The {\em eccentricity} of a vertex $v$ of a connected graph $G$ is the maximum  distance between $v$ and any other vertex $u$ of G. The {\em radius} of $G$ is the minimum  eccentricity of any vertex in $G$. It was shown in \cite{CartesianIndependence} that for every connected graphs $G$ and $H$ of radius $r(G)$ and $r(H)$, it follows that
\begin{equation}\label{CotaIndepCartesian}
\alpha(G\Box H)\ge 2r(G)r(H).
\end{equation}
As a direct consequence of (\ref{CotaOffInd}) and (\ref{CotaIndepCartesian}) we have the following bound.
\begin{proposition}\label{cotasuperior}
For any connected graphs $G$ and $H$ of order $n_1$ and $n_2$ and radius $r(G)$ and $r(H)$, respectively,
$$\gamma_o(G\Box H)\le n_1n_2- 2r(G)r(H).$$
\end{proposition}

The above bound is tight. It is achieved, for instance, for the torus graphs $C_{2k}\Box C_{2k'}$ (see Proposition \ref{CiclosporCiclos}), for the grid graphs $P_{2k}\Box P_{2k'}$ (see Proposition \ref{ProdPnxPn}) and for the cylinder graphs $P_{2k}\Box C_{2k'}$ (see Proposition \ref{cilindros}).

We derive another upper bound on $\gamma_o(G\Box H)$ from (\ref{CotaOffInd}) and the following bound on $\alpha(G\Box H)$ obtained in \cite{Imrich} for every bipartite graph $G$ of order $n$:
  \begin{equation}\label{CotaIndepCartesian2}
\alpha(G\Box H)\ge \frac{n}{2}\alpha_2(H),
\end{equation}
where $\alpha_2(H)$ is the bipartite number  of $H$, i.e., the order of the largest
induced bipartite subgraph of $H$.
\begin{proposition}\label{cotasuperior2}
For any connected bipartite graph $G$ of order $n_1$ and any connected graph $H$ of order $n_2$ and bipartite number $\alpha_2(H)$,
$$\gamma_o(G\Box H)\le n_1\left(n_2- \frac{\alpha_2(H)}{2}\right).$$
\end{proposition}

As direct consequence of Proposition \ref{cotasuperior2} we obtain the following result.

\begin{corollary}\label{bipartitos}
Let $G$ and  $H$ be two connected bipartite graphs  of order $n_1$ and $n_2$, respectively.  Then  $ \gamma_o(G\Box H)\le  \frac{n_1n_2}{2}$.
\end{corollary}

The above bound is tight. It is achieved, for instance, for the torus graphs $C_{2k}\Box C_{2k'}$ (see Proposition \ref{CiclosporCiclos}), for the grid graphs $P_{2k}\Box P_{2k'}$ (see Proposition \ref{ProdPnxPn}) and for the cylinder graphs $P_{2k}\Box C_{2k'}$ (see Proposition \ref{cilindros}).

We recall that a graph $H=(V,E)$ is partitionable into two global offensive alliances if there exists a partition $\{Y_1,Y_2\}$ of $V$ such that both $Y_1$ and $Y_2$ are global offensive alliances in $H$ \cite{yero1}.

\begin{theorem}\label{partition-two-global}
Let $H$ be a graph of order $n$.  If $H$ is partitionable into two global offensive alliances, then $$\gamma_o(K_{r}\Box H)\le \left\lfloor\frac{rn}{2}\right\rfloor.$$
\end{theorem}

\begin{proof}
Let $U=\{u_1,u_2,...,u_{r}\}$  and $V=\{v_1,v_2,...,v_{n}\}$ be the sets of vertices of $K_r$ and $H$, respectively.
Let  $\{Y_1,Y_2\}$ be  a partition of $V$ such that both $Y_1$ and $Y_2$ are global offensive alliances in $H$, where $|Y_1|\le |Y_2|$. Let $X_1=\{u_1,u_2,...,u_{\left\lceil\frac{r}{2}\right\rceil}\}$ and let
$X_2=\{u_{\left\lceil\frac{r}{2}\right\rceil+1},u_{\left\lceil\frac{r}{2}\right\rceil+2},...,u_r\}$. Let us show that $S=(X_1\times Y_1)\cup (X_2\times Y_2)$ is a global offensive alliance in $K_r\Box H$.
If $(u,v)\in X_1\times Y_2$, then
\begin{align*}
   \delta_S(u,v)&=\delta_{X_2}(u)+\delta_{Y_1}(v)\\
   &=\left\lfloor\frac{r}{2}\right\rfloor+\delta_{Y_1}(v)\\
   &\ge \left\lceil\frac{r}{2}\right\rceil-1+\delta_{Y_1}(v)\\
   &\ge \left\lceil\frac{r}{2}\right\rceil-1+\delta_{Y_2}(v)+1\\
   &=\delta_{X_1}(u)+\delta_{Y_2}(v)+1\\
   &=\delta_{\overline{S}}(u,v)+1.
 \end{align*}

Now, if $(u,v)\in X_2\times Y_1$, then
\begin{align*}
   \delta_S(u,v)&=\delta_{X_1}(u)+\delta_{Y_2}(v)\\
   &=\left\lceil\frac{r}{2}\right\rceil+\delta_{Y_2}(v)\\
   &\ge \left\lceil\frac{r}{2}\right\rceil+\delta_{Y_1}(v)+1\\
   &\ge \left\lfloor\frac{r}{2}\right\rfloor-1+\delta_{Y_1}(v)+1\\
   &=\delta_{X_2}(u)+\delta_{Y_1}(v)+1\\
   &=\delta_{\overline{S}}(u,v)+1.
 \end{align*}
 Hence,  $S=(X_1\times Y_1)\cup (X_2\times Y_2)$ is a global offensive alliance in $K_r\Box H$. If $r$ is even, then $|S|=\frac{rn}{2}$ and, if $r$ is odd, then
  $|S|=\frac{r+1}{2} |Y_1| + \frac{r-1}{2}  | Y_2|=\frac{r}{2}n  + \frac{|Y_1|-|Y_2|}{2}\le \frac{r}{2}n$.
 The proof is complete.
\end{proof}

\begin{proposition}\label{bamboo}
The global offensive alliance number of the bamboo graph $K_r\Box P_t$ is
$$\gamma_o(K_r\Box P_t)=\left\lfloor\frac{rt}{2}\right\rfloor.$$
\end{proposition}

\begin{proof}
Since $P_t$ is partitionable into two global offensive alliances, by Theorem \ref{partition-two-global} we obtain the upper bound $\gamma_o(K_r\Box P_t)\le \left\lfloor\frac{rt}{2}\right\rfloor.$

On the other hand, let $S$ be a $\gamma_o(K_r\Box P_t)$-set. Let $\{u_1,u_2,...,u_{r}\}$  and $\{v_1,v_2,...,v_{t}\}$ be the sets of vertices of $K_r$ and $P_t$, respectively. In $P_t$ adjacent vertices have consecutive subscripts.
 Now, let $V_j$ be the vertex set of the $v_j$-fiber and  let $S_j=P_{K_r}(S\cap V_j)$.  We first note that since $S$ is a global offensive alliance in $K_r\Box P_t$,  for every $(u_i,v_1)\in V_1$ it follows
   $1+|S_1|\ge \delta_S(u_i,v_1)\ge \frac{r+1}{2}$, so $|S_1|\ge \left\lceil\frac{r-1}{2}\right\rceil=\left\lfloor\frac{r}{2}\right\rfloor$. Analogously,  $|S_t|\ge \left\lfloor\frac{r}{2}\right\rfloor$. Now, suppose there exists $j\in \{1,...,t-1\}$ such that $|S_j\cup S_{j+1}|< r$. In such a case, $|S_j|\le \frac{r-1}{2}$ or $|S_{j+1}|\le \frac{r-1}{2}$, and there exist $(u_i,v_j),(u_i,v_{j+1})\not\in S.$ We take, without loss of generality, $|S_j|\le \frac{r-1}{2}$. Thus,
 $$1+\frac{r-1}{2}\ge \delta_S(u_i,v_j)\ge \delta_{\overline{S}}(u_i,v_j)+1\ge \frac{r-1}{2}+2,$$
 a contradiction. Hence,  for every $j\in \{1,2,...,t-1\}$, $|S_j\cup S_{j+1}|\ge r$.  As a consequence
$$2|S|=\sum_{j=1}^{{t}}| S_j|\ge r (t-1)+|S_1|+|S_t|\ge r (t-1)+2\left\lfloor\frac{r}{2}\right\rfloor.$$
Thus, $$\gamma_o(K_r\Box P_t)\ge  \left\lceil\frac{r (t-1)}{2}\right\rceil+\left\lfloor\frac{r}{2}\right\rfloor=\left\lfloor\frac{rt}{2}\right\rfloor.$$
Therefore, the proof is complete.
\end{proof}

\begin{corollary}
For any complete graph $K_r$ and any path graph $P_t$,
$$\gamma_o(K_{r}\Box P_{t})\ge \gamma_o(K_r)\gamma_o(P_t).$$
\end{corollary}

We recall that a set $S$ of vertices of a graph $H$ is a global strong offensive alliance if for every $x\in \overline{S}$, $\delta_S(x)\ge \delta_{\overline{S}}(x)$.
Note that every global offensive alliance is a global strong offensive alliance. It was shown in \cite{yero1} that every graph without isolated vertices is partitionable into global strong offensive alliances.

\begin{theorem}\label{partitionable}
Let $G$  be a graph partitionable into two global offensive alliances $X_1$ and $ X_2$ and let $H$ be a graph partitionable into two global strong offensive alliances $Y_1$ and $Y_2$. Then
$$\gamma_o(G\Box H)\le  |X_1||Y_1|+|X_2||Y_2|.$$
\end{theorem}

\begin{proof}
 Let $U=\{u_1,u_2,...,u_{n_1}\}$  and $V=\{v_1,v_2,...,v_{n_2}\}$ be the sets of vertices of $G$ and $H$, respectively, where
  $X_1=\{u_1,u_2,...,u_{|X_1|}\}$,
 $X_2=\{u_{|X_1|+1},u_{|X_1|+2},...,u_{n_1}\}$, $Y_1=\{v_{1},v_{2},...,v_{|Y_1|}\}$  and    $Y_2=\{v_{|Y_1|+1},v_{|Y_1|+2},...,v_{n_2}\}$. Let us show that $S=(X_1\times Y_1)\cup (X_2\times Y_2)$ is a global offensive alliance in $G\Box H$.
If $(u,v)\in X_1\times Y_2$, then
\begin{align*}
   \delta_S(u,v)&=\delta_{X_2}(u)+\delta_{Y_1}(v)\\
   &\ge\delta_{X_1}(u)+\delta_{Y_2}(v)+1\\
   &\ge\delta_{\overline{S}}(u,v)+1.
 \end{align*}

Now, if $(u,v)\in X_2\times Y_1$, then
\begin{align*}
   \delta_S(u,v)&=\delta_{X_1}(u)+\delta_{Y_2}(v)\\
   &\ge \delta_{X_2}(u)+\delta_{Y_1}(v)+1\\
   &\ge \delta_{\overline{S}}(u,v)+1.
 \end{align*}
 Hence,  $S=(X_1\times Y_1)\cup (X_2\times Y_2)$ is a global offensive alliance in $G\Box H$.
 The proof is complete.
\end{proof}

The proof of the following result is completely analogous to the proof of  Theorem \ref{partitionable}.
\begin{theorem}\label{partition-glob-strong}
Let $G$  be a graph partitionable into a global offensive alliance $X_1$ and a global strong offensive alliance $ X_2$.  Let $H$ be a graph partitionable into a  global offensive alliance $Y_1$ and a global strong offensive alliance $Y_2$. Then
$$\gamma_o(G\Box H)\le  |X_1||Y_1|+|X_2||Y_2|.$$
\end{theorem}

A bipartite graph $G=(X_1\cup X_2, E)$, where the sets of the bipartition have cardinality $|X_1|=x_1$ and $|X_2|=x_2$ is called a $(x_1,x_2)$-bipartite graph.

\begin{corollary}
Let $G$ be a $(p_1,p_2)$-bipartite graph and let $H$ be a $(t_1,t_2)$-bipartite graph. Then
$$\gamma_o(G\Box H)\le  p_1t_1+p_2t_2.$$
\end{corollary}

We recall that the hypercube graphs are defined as $Q_k=Q_{k-1}\Box K_2$, $k\ge 2$, where $Q_1=K_2$. Note that $Q_{k-1}$ is a $(2^{k-2},2^{k-2})$-bipartite graph and $K_2$ is a $(1,1)$-bipartite graph. Moreover, the Laplacian spectral radius of $Q_k$ is $\lambda=2k$. Hence,    from the above corollary and (\ref{lemma-cota-glob}) we have $$\left\lceil\left\lceil\frac{k+1}{2}\right\rceil \frac{2^{k-1}}{k}\right\rceil\le \gamma_o(Q_k)\le 2^{k-1}.$$


\begin{proposition}\label{CiclosporCiclos}
The global offensive alliance number of the torus graph $C_r\Box C_t$ is $$\gamma_o(C_{r}\Box C_{t})=\left\lceil\frac{rt}{2}\right\rceil.$$
\end{proposition}

\begin{proof}
Since every cycle $C_n$ can be partitioned into a global strong offensive alliance of cardinality $\left\lfloor\frac{n}{2}\right\rfloor$ and a  global offensive alliance of cardinality $\left\lceil\frac{n}{2}\right\rceil$, by Theorem \ref{partition-glob-strong} we have  $\gamma_o(C_{r}\Box C_{t})\le \left\lceil\frac{rt}{2}\right\rceil.$

On the other hand, let $S$ be a $\gamma_o(C_r\Box C_t)$-set. Let $\{u_0,u_1,...,u_{r-1}\}$  and $\{v_0,v_1,...,v_{t-1}\}$ be the sets of vertices of $C_r$ and $C_t$, respectively. Here adjacent vertices have consecutive subscripts, where the subscripts are taken modulo $r$ and $t$, respectively. As above, let $V_j$ be the vertex set of the $v_j$-fiber and  let $S_j=P_{C_r}(S\cap V_j)$.  Let $(u_i,v_j)$ be a vertex not belonging to $S$. Since  $C_r \Box C_t$ is a $4$-regular graph and $S$ is a global offensive alliance, if $(u_{i+1},v_j)\not\in S$, then $(u_i,v_{j+1}), (u_{i+1},v_{j+1})\in S$, and if  $(u_i,v_{j+1})\not\in S$, then $(u_{i+1},v_j), (u_{i+1},v_{j+1})\in S$. Thus,  for every $j\in \{0,1,...,t-1\}$, $|S_j\cup S_{j+1}|\ge r$.  Hence,
$$2|S|=\sum_{j=0}^{{t-1}}| S_j\cup S_{j+1})|\ge r t.$$
Therefore, we have that $\gamma_o(C_r\Box C_t)\ge \left\lceil\frac{rt}{2}\right\rceil$ and the proof is complete.
\end{proof}

\begin{corollary}
For any torus graph $C_r\Box C_t$,   $\gamma_o(C_{r}\Box C_{t})\ge \gamma_o(C_r)\gamma_o(C_t).$
\end{corollary}


\begin{proposition}\label{CompletosporCiclos}
The global offensive alliance number of the graph $K_r\Box C_t$ is $$\gamma_o(K_{r}\Box C_{t})=\left\lceil\frac{rt}{2}\right\rceil.$$
\end{proposition}

\begin{proof}
Let $S$ be a $\gamma_o(K_r\Box C_t)$-set. Let $\{u_1,u_2,...,u_{r}\}$  and $\{v_0,v_1,...,v_{t-1}\}$ be the sets of vertices of $K_r$ and $C_t$, respectively. In $C_t$ the subscripts are taken modulo  $t$ and adjacent vertices have consecutive subscripts. As above, let $V_j$ be the vertex set of the $v_j$-fiber and  let $S_j=P_{K_r}(S\cap V_j)$.  Let $(u_i,v_j)$ be a vertex not belonging to $S$.

Now, suppose there exists $j\in \{0,1,...,t\}$ such that $|S_j\cup S_{j+1}|< r$. In such a case, $|S_j|\le \frac{r-1}{2}$ or $|S_{j+1}|\le \frac{r-1}{2}$, and there exist $(u_i,v_j),(u_i,v_{j+1})\not\in S.$ We take, without loss of generality, $|S_j|\le \frac{r-1}{2}$. Thus,
 $$1+\frac{r-1}{2}\ge \delta_S(u_i,v_j)\ge \delta_{\overline{S}}(u_i,v_j)+1\ge \frac{r-1}{2}+2,$$
 a contradiction. Hence,  for every $j\in \{1,2,...,t-1\}$, $|S_j\cup S_{j+1}|\ge r$ (the subscripts are taken modulo $t$).  As a consequence
$2|S|=\sum_{j=0}^{{t-1}}| S_j\cup S_{j+1}|\ge rt.$
Therefore, $\gamma_o(K_r\Box C_t)\ge  \left\lceil\frac{r t}{2}\right\rceil.$

Since every cycle graph $C_n$  (every complete graph $K_n$) can be partitioned into a global strong offensive alliance of cardinality $\left\lfloor\frac{n}{2}\right\rfloor$ and a  global offensive alliance of cardinality $\left\lceil\frac{n}{2}\right\rceil$, by Theorem \ref{partition-glob-strong} we have  $\gamma_o(K_{r}\Box C_{t})\le \left\lceil\frac{rt}{2}\right\rceil.$
Therefore, the proof is complete.
\end{proof}

\begin{corollary}
For any complete graph $K_r$ and any cycle graph $C_t$,
$$\gamma_o(K_{r}\Box C_{t})\ge \gamma_o(K_r)\gamma_o(C_t).$$
\end{corollary}

 A {\em square} in a Cartesian product of two graphs $G$ and $H$ is a set of vertices of $G\Box H$ formed by four different vertices $(u_i,v_k)$, $(u_i,v_l)$, $(u_j,v_k)$, $(u_j,v_l)$ such that $u_i\sim u_j$ in $G$ and $v_k\sim v_l$ in $H$.

\begin{lemma}\label{remark-4-2}
Let $G$ and $H$ be two graphs such that they are cycles or paths and let $S$ be a $\gamma_o(G\Box H)$-set. For any square $A$ of $G\Box H$ it follows that $|S\cap A|\ge 2$.
\end{lemma}

\begin{proof}
The result follows directly from the fact that if at least three of the vertices of the square $A=\{(u_i,v_k)$, $(u_i,v_l)$, $(u_j,v_k)$, $(u_j,v_l)\}$ do not belong to $S$, then (at least) for one of these three vertices, say $(u_i,v_k)$, it is satisfied that $\delta_{\overline{S}}(u_i,v_k)\ge 2$, which is a contradiction because $G\Box H$ has maximum degree four and by (\ref{alliaceCondition2}) we know that for every  $(u,v)\in \overline{S}$, it follows $\delta_{\overline{S}}(u,v)\le 1$.
\end{proof}

\begin{proposition}\label{ProdPnxPn}
Let $P_r\Box P_t$ be a grid graph.
\begin{enumerate}[{\rm (i)}]
\item If $r$ and $t$ are even, then $\gamma_o(P_r\Box P_t)= \frac{rt}{2}$.

\item If $r$ is even and $t$ is odd, then $\gamma_o(P_r\Box P_t)= \frac{r(t-1)}{2}+\left\lceil\frac{r}{3}\right\rceil$.

\item If $r$  and $t$ are odd, then $\frac{(r-1)(t-1)}{2}+\left\lceil\frac{r}{3}\right\rceil+\left\lceil\frac{t}{3}\right\rceil \le \gamma_o(P_r\Box P_t)\le  \frac{r(t-1)}{2}+\left\lceil\frac{r}{3}\right\rceil$.
\end{enumerate}
\end{proposition}

\begin{proof}
 Let $V_1=\{u_1,u_2,...,u_r\}$ and  $V_2=\{v_1,v_2,...,v_t\}$ be the sets of vertices of   $P_r$ and $P_t$, respectively. Here adjacent vertices have consecutive subscripts.  Let $S$ be a $\gamma_o(P_r\Box P_t)$-set.

Suppose $r$ and $t$ are even. Since there exists a vertex partition of $P_r\Box P_t$ into $\frac{rt}{4}$ disjoint squares, by Lemma \ref{remark-4-2} we have that
$\gamma_o(P_r\Box P_t)=|S|\ge \frac{rt}{2}.$ Moreover, by Proposition \ref{cotasuperior} we have $\gamma_o(P_r\Box P_t)=|S|\le \frac{rt}{2}.$ Therefore, (i) follows.\\

Now we suppose $r$ is even and  $t$ is odd. Since there exists a vertex partition of $P_r\Box P_{t-1}$ into $\frac{r(t-1)}{4}$ disjoint squares, by Lemma \ref{remark-4-2} we have that $\gamma_o(P_r\Box P_t)\ge \frac{r(t-1)}{2} +|S_t|$.

As above, let $V_t$ be the set of vertices of the $P_r$-fiber corresponding to the vertex $v_t$ of $P_t$ and let $S_t=S\cap V_t$. Notice that if a vertex of $V_t$, of degree two in $P_r\Box P_t$, does not belong to $S$, then  its two neighbors belong to $S$. Also, if three vertices of $V_r$ induce a path in $P_r\Box P_t$, then at least one of them belongs to $S$. Thus,  we have that $|S_t|\ge \left\lceil\frac{r}{3}\right\rceil$ and, as a consequence,
$\gamma_o(P_r\Box P_t)\ge \frac{r(t-1)}{2}+|S_t|\ge \frac{r(t-1)}{2}+\left\lceil\frac{r}{3}\right\rceil.$

On the other hand, let $W$ be the subset of vertices of $P_r\Box P_t$ taken in the following way:\\
\noindent
If $r\equiv 0 \pmod 3$,  then $W$ is composed of pairs $(u_i,v_j)$ with $i\in \{2,5,8,...,r-1\}$ and $j\in \{1,3,5,...,t-2,t\}$  as well as pairs $(u_i,v_j)$ with   $i\in \{1,3,4,6,7,...,r-3,r-2,r\}$ and $j\in \{2,4,6,...,t-3,t-1\}$.\\

Notice that in this case $$|W|=\frac{t-1}{2}\frac{r}{3}+\frac{t-1}{2}\left(r-\frac{r}{3}\right)+\frac{r}{3}=\frac{r(t-1)}{2}+\frac{r}{3}=\frac{r(t-1)}{2}+\left\lceil\frac{r}{3}\right\rceil.$$

\noindent
If $r\equiv 1 \pmod 3$,  then $W$ is composed of pairs $(u_i,v_j)$ with  $i\in \{1,4,7,...,r-6,r-3,r\}$ and $j\in \{1,3,5,...,t-2,t\}$ as well as pairs $(u_i,v_j)$ with $i\in \{2,3,5,6,...,r-2,r-1\}$ and $j\in \{2,4,6,...,t-3,t-1\}$.\\

If $r\equiv 2 \pmod 3$,  then $W$ is composed of pairs $(u_i,v_j)$ with $i\in \{1,4,7,...,r-7,r-4,r-1\}$ and $j\in \{1,3,5,...,t-2,t\}$  as well as pairs $(u_i,v_j)$ with $i\in \{2,3,5,6,...,r-3,r-2,r\}$ and $j\in \{2,4,6,...,t-3,t-1\}$.\\

Thus, in the above two cases ($r\equiv 1 \pmod 3$ or $r\equiv 2 \pmod 3$) we have
$$|W|=\frac{t-1}{2}\left\lceil\frac{r}{3}\right\rceil+\frac{t-1}{2}\left(r-\left\lceil\frac{r}{3}\right\rceil\right)+\left\lceil\frac{r}{3}\right\rceil=
\frac{r(t-1)}{2}+\left\lceil\frac{r}{3}\right\rceil.$$
So, in all the above cases we have that $|W|=\frac{r(t-1)}{2}+\left\lceil\frac{r}{3}\right\rceil$. Moreover, for every vertex $(u,v)\notin W$ we have that $\delta_{W}(u,v)=2\ge 1=\delta_{\overline{W}}(u,v)+1$, if $(u,v)$ has degree two, $\delta_{W}(u,v)=3\ge 1=\delta_{\overline{W}}(u,v)+1$, if $(u,v)$ has degree three,  and $\delta_{W}(u,v)=3\ge 2=\delta_{\overline{W}}(u,v)+1$, if $(u,v)$ has degree four. Thus, $W$ is a global offensive alliance and so, $\gamma_o(P_r\Box P_t)\le \frac{r(t-1)}{2}+\left\lceil\frac{r}{3}\right\rceil.$ Therefore, (ii) follows.\\

Finally, we suppose $r$ and  $t$ are  odd. Since there exists a vertex partition of $P_{r-1}\Box P_{t-1}$ into $\frac{(r-1)(t-1)}{4}$ disjoint squares,  using a similar argument to the above case we have that
$\gamma_o(P_r\Box P_t)\ge  \frac{(r-1)(t-1)}{2}+\left\lceil\frac{t}{3}\right\rceil+\left\lceil\frac{r}{3}\right\rceil.$
On the other hand, let $Q$ be the subset of vertices of $P_r\Box P_t$ taken in the following way:\\

\noindent
If $r\equiv 0 \pmod 3$, then $Q$ is composed of pairs $(u_i,v_j)$ with $i\in \{2,5,8,...,r-1\}$ and $j\in \{1,3,5,...,t-2,t\}$  as well as pairs $(u_i,v_j)$ with $i\in \{1,3,4,6,7,...,r-3,r-2,r\}$ and $j\in \{2,4,6,...,t-3,t-1\}$. \\

\noindent
If $r \equiv 1 \pmod 3$, then $Q$ is composed of pairs $(u_i,v_j)$ with $i\in \{1,4,7,...,r-6,r-3,r\}$ and $j\in \{1,3,5,...,t-2,t\}$ as well as pairs $(u_i,v_j)$ with $i\in \{2,3,5,6,...,r-2,r-1\}$ and $j\in \{2,4,6,...,t-3,t-1\}$. \\

 \noindent
If $r \equiv 2 \pmod 3$, then $Q$ is composed of pairs $(u_i,v_j)$ with $i\in \{1,4,7,...,r-7,r-4,r-1\}$ and $j\in \{1,3,5,...,t-2,t\}$  as well as pairs $(u_i,v_j)$ with $i\in \{2,3,5,6,...,r-3,r-2,r\}$ and $j\in \{2,4,6,...,t-3,t-1\}$. \\

Thus, in the above cases  we have
$$|Q|=\left\lceil\frac{r}{3}\right\rceil \left\lceil\frac{t}{2}\right\rceil+\left(r-\left\lceil\frac{r}{3}\right\rceil \right) \left\lfloor\frac{t}{2}\right\rfloor= \frac{r(t-1)}{2}+\left\lceil\frac{r}{3}\right\rceil .$$
 Moreover, for every vertex $(u,v)\notin Q$ we have that  $\delta_{Q}(u,v)=2\ge 1=\delta_{\overline{Q}}(u,v)+1$, if $(u,v)$ has degree two, $\delta_{Q}(u,v)=3\ge 1=\delta_{\overline{Q}}(u,v)+1$, if $(u,v)$ has degree three,  and $\delta_{Q}(u,v)=3\ge 2=\delta_{\overline{Q}}(u,v)+1$, if $(u,v)$ has degree four.  Thus, $Q$ is a global offensive alliance in $P_r\Box P_t$ and, as a consequence, $\gamma_o(P_r\Box P_t)\le \frac{r(t-1)}{2}+\left\lceil\frac{r}{3}\right\rceil $. Therefore, (iii) follows.
\end{proof}

\begin{corollary}
For any path graph $P_r$ and any path graph  $P_t$, $$\gamma_o(P_{r}\Box P_{t})\ge \gamma_o(P_r)\gamma_o(P_t).$$
\end{corollary}

\begin{proposition}\label{cilindros}
The global offensive alliance number of the cylinder graph $P_r\Box C_t$ is
$$\gamma_o(P_r\Box C_t)=\left\{\begin{array}{ll}
            \frac{rt}{2}, & \mbox{if $r$ is even,} \\
            & \\
            \frac{(r-1)t}{2}+\left\lceil\frac{t}{3}\right\rceil, & \mbox{if $r$ is odd.}
                                                                                   \end{array}
\right.$$
\end{proposition}

\begin{proof}
Let $S$ be a $\gamma_o(P_r\Box C_t)$-set. Let $\{u_1,u_2,...,u_{r}\}$  and $\{v_0,v_1,...,v_{t-1}\}$ be the sets of vertices of $P_r$ and $C_t$, respectively (Here adjacent vertices have consecutive subscripts. In the case of $C_t$, the subscripts are taken modulo $t$). We differentiate the following cases.\\

\noindent
Case 1: $r$ even.
 As above, let $V_j$ be the vertex set of the $v_j$-fiber and  let $S_j=P_{P_r}(S\cap V_j)$.  Let $(u_i,v_j)$ be a vertex not belonging to $S$. Since every vertex of $P_r \Box C_t$ has degree three or four and $S$ is a global offensive alliance, if $(u_{i+1},v_j)\not\in S$, then $(u_i,v_{j+1}), (u_{i+1},v_{j+1})\in S$, and if  $(u_i,v_{j+1})\not\in S$, then $(u_{i+1},v_j), (u_{i+1},v_{j+1})\in S$. Thus,  for every $j\in \{0,1,...,t-1\}$, $|S_j\cup S_{j+1}|\ge r$.  Hence
$$2|S|=\sum_{j=0}^{{t-1}}| S_j\cup S_{j+1}|\ge r t.$$
Therefore, we have that $\gamma_o(P_r\Box C_t)\ge \frac{rt}{2}$.
Since every cycle $C_n$ (every path $P_n$) can be partitioned into a global strong offensive alliance of cardinality $\left\lfloor\frac{n}{2}\right\rfloor$ and a  global offensive alliance of cardinality $\left\lceil\frac{n}{2}\right\rceil$, by Theorem \ref{partition-glob-strong} we have  $\gamma_o(P_{r}\Box C_{t})\le \frac{rt}{2}.$

\noindent
Case 2:  $r$ odd.
The number of squares of $P_r\Box C_t$ is $(r-1)t$. By Lemma \ref{remark-4-2} we know that each square of $P_r\Box C_t$ contains at least two vertices belonging to $S$, moreover,  each vertex of $S$ belongs to four different squares, except the vertices of degree three which only belong to two different squares. So, we have $2(r-1)t\le 4(|S|-|S'|)+2|S'|$, where $S'=\{(u,v)\in S\,:\,d(u,v)=3\}$. Note also that if three vertices of degree three induce a path in $P_r\Box C_t$, then at least one of them belongs to $S$.  Thus, $|S'|\ge 2\left\lceil\frac{t}{3}\right\rceil$. Hence,
$|S|\ge \frac{(r-1)t}{2}+\left\lceil\frac{t}{3}\right\rceil.$

Now, let $Y$ be the subset of vertices of $P_r\Box C_t$ which is formed by pairs $(u_i,v_j)$ with
 $i\in \{1, 3,5,...,r\}$ and $j\in \{k\,:\,0\le k\le t-1,\, k\equiv 0\; (\mbox{mod}\; 3)\}$ as well as pairs $(u_i,v_j)$ with
 $i\in \{2,4,...,r-1\}$ and $j\in \{k\,:\,0\le k\le t-1,\, k\not\equiv 0\; (\mbox{mod}\; 3)\}$.

Then, clearly $|Y|=\left\lceil\frac{r}{2}\right\rceil \left\lceil\frac{t}{3}\right\rceil+ \left\lfloor\frac{r}{2}\right\rfloor (t-\left\lceil\frac{t}{3}\right\rceil)=\frac{(r-1)t}{2}+\left\lceil\frac{t}{3}\right\rceil$. Now, since for every  $(u,v)\notin Y$ we have that $\delta_{\overline{Y}}(u,v)= 0$ or $\delta_{\overline{Y}}(u,v)= 1$, we conclude that  $Y$ is a global offensive alliance in $P_r\Box C_t$ and, as a consequence, $\gamma_o(P_r\Box C_t)\le |Y|=\frac{(r-1)t}{2}+\left\lceil\frac{t}{3}\right\rceil$.
Therefore, the proof is complete.



\end{proof}

\begin{corollary}
For any path graph $P_r$ and any cycle graph $C_t$,
$$\gamma_o(P_{r}\Box C_{t})\ge \gamma_o(P_r)\gamma_o(C_t).$$
\end{corollary}

It was shown in  \cite{cota-global} that the global offensive alliance number of a connected graph $G$ of order $n$ is bounded by
 \begin{equation}\label{lemma-cota-glob}
 \gamma_o(G)\ge \left\lceil\frac{n}{\lambda}\left\lceil\frac{\delta+1}{2}\right\rceil  \right\rceil,
\end{equation}
where $\lambda$ is the Laplacian spectral radius\footnote{The Laplacian spectral radius of a graph $G$ is the largest laplacian eigenvalue of $G$. More information about Laplacian eigenvalues of a graph can be found in \cite{laplacian-mohar}.}  of $G$ and  $\delta$ its minimum degree. This bound will be useful to prove the following result.

\begin{proposition}\label{CompletoxCompleto}
Let $r$ and $t$ be two positive integers. If $r,t$ have the same parity, then
$$\gamma_o(K_{r}\Box K_{t})= \left\lceil\frac{rt}{2}\right\rceil.$$
If $r$ and $t$ have different parity, then
$$\left\lceil\frac{rt(r+t-1)}{2(r+t)}\right\rceil   \le \gamma_o(K_{r}\Box K_{t})\le  \left\lceil\frac{rt}{2}\right\rceil.$$
\end{proposition}

\begin{proof}
 Since every complete graph $K_n$ can be partitioned into a global strong offensive alliance of cardinality $\left\lfloor\frac{n}{2}\right\rfloor$ and a  global offensive alliance of cardinality $\left\lceil\frac{n}{2}\right\rceil$, by Theorem \ref{partition-glob-strong} we have  $\gamma_o(K_{r}\Box K_{t})\le \left\lceil\frac{rt}{2}\right\rceil.$

On the other hand, in order to apply (\ref{lemma-cota-glob}) to  $K_{r}\Box K_{t}$, we recall that in this case we have order $rt$, degree $r+t-2$ and Laplacian spectral radius $\lambda=r+t$. So, if $r$ and $t$ have the same parity, then (\ref{lemma-cota-glob}) leads to $\gamma_o(K_{r}\Box K_{t})\ge  \left\lceil\frac{rt}{2}\right\rceil,$  and if $r$ and $t$ have different parity, then (\ref{lemma-cota-glob}) leads to $\gamma_o(K_{r}\Box K_{t})\ge \left\lceil\frac{rt(r+t-1)}{2(r+t)}\right\rceil$.
The proof is complete.
\end{proof}

\begin{corollary}
For any complete graphs
$$\gamma_o(K_r\Box K_t)\ge \gamma_o(K_r)\gamma_o(K_t).$$
\end{corollary}

\section*{Acknowledgements}
The authors would like to thank the anonymous reviewers
for their valuable comments and suggestions to improve the
quality of the paper.


\end{document}